\documentclass[12pt]{article}
\usepackage{amssymb,amsfonts,amsmath,amsthm,dsfont,hyperref}
\usepackage[lite,alphabetic]{amsrefs}
\usepackage[letterpaper, portrait, margin=26.5mm]{geometry}

\newcommand{\1}{\mathds{1}}

\newcommand{\Q}{\mathbb{Q}}
\newcommand{\R}{\mathbb{R}}

\newcommand{\N}{\mathbb{N}}

\newcommand{\B}{\mathrm{B}}

\newcommand{\Bbo}{\overline{\mathrm{B}}}
\newcommand{\8}{\infty}

\newcommand{\conv}{\mathrm{conv}}

\newcommand{\Po}{\mathcal{P}}

\newcounter{dummy} \numberwithin{dummy}{section}
\newtheorem{theorem}[dummy]{Theorem}
\newtheorem{lemma}[dummy]{Lemma}

\newtheorem{proposition}[dummy]{Proposition}

\theoremstyle{remark}
\newtheorem{remark}[dummy]{Remark}
\newtheorem{example}[dummy]{Example}
\usepackage{scalerel}
\DeclareMathOperator*{\bigplus}{\scalerel*{+}{\sum}}

\begin{document}

\title{Locally hulled topologies}
\author{Eugene Bilokopytov\footnote{Email address bilokopy@ualberta.ca, erz888@gmail.com.}}
\maketitle

\begin{abstract}
We present a general result about generating group topologies by pseudo-norms. Namely, we show that if a topology has a base of sets which are closed in a certain sense, then it can be generated by a collection of pseudo-norms such that the balls in these pseudo-norms are also closed in the same sense. The examples include linear and locally convex topologies on vector spaces, locally solid and Fatou topologies on vector lattices and Fr\'echet-Nikod\'ym topologies on Boolean algebras.


MSC2020 22A20, 46A40, 52A01
\end{abstract}

\section{Introduction}

This note is a side branch of the ongoing project on locally solid convergence structures on vector lattices (see e.g. \cite{erz} and \cite{ectv}), and is a precursor to the work \cite{erzh}. In the study of locally solid topologies on a vector lattice (see e.g. \cite{ab0}, \cite{fremlin}) instead of working with topologies directly one can look at solid pseudo-norms which generate these topologies. Similarly, the study of locally convex topologies can be partially reduced to the semi-norms on a vector space. An analogous situation occures in the realm of Boolean algebras, where one can study Fr\'echet-Nikod\'ym topologies (see \cite{weber}), or the submeasures. In all these cases, we rely on the fact that the corresponding topology is completely determined by a family of functionals of a certain type. This paper is dedicated to a general result which would include the aforementioned ones as examples.\medskip

Namely, we will consider group topologies on commutative groups which have local bases at $0$ consisting of sets which are closed in a certain sense. The ``certain sense'' part is formalized via the notion of a hull structure. This concept exists in the literature under different names (see e.g. \cite{dp} and \cite{erne}), and is a generalization of the notion of a topology. Convex sets, solid sets and many others are examples of hull structures. We call the corresponding topologies locally hulled, keeping with the terminology of ``locally convex / solid''. We prove (Theorem \ref{generated}) that for a wide class of hull structures, any locally hulled group topology is generated by a collection of pseudo-norms all of whose balls are hulled. After that we explain how to deduce some classical facts in more specific contexts.\medskip

The novelty of the approach developed in the paper lies in its high level of generality and systematic use of the hull structures. The exposition is mostly self contained except for the examples.

\section{Hull structures}\label{hull}

A \emph{hull structure} on a set $X$ is a collection $\mathcal{Q}$ of subsets of $X$ which is closed with respect to arbitrary intersections and contains $X$. It generates a \emph{hull operator} from $\Po\left(X\right)$ into $\mathcal{Q}$ defined by $A^{\mathcal{Q}}:=\bigcap\left\{Q\in\mathcal{Q},~Q\supset A\right\}\in\mathcal{Q}$. Examples include closed sets in a topological space, convex sets in a vector space, closed convex sets in a topological vector space, subgroups of a group, ideals in a ring, order closed sets in a vector lattice (see Example \ref{oclosed}) etc.

Additionally, let us mention two examples of hull structures on posets. If $\le$ is a partial order on $X$, call $Q\subset X$ \emph{full} if $p,q\in Q$, $p\le r\le q$ $\Rightarrow$ $r\in Q$, and a \emph{lower set} if $p\le q\in Q$ $\Rightarrow$ $p\in Q$. Clearly, every lower set is full, and both full and lower sets form hull structures on $X$.

Let $\mathcal{Q},\mathcal{R}$ be hull structures on $X$ and $Y$. A map $\varphi:X\to Y$ is called \emph{continuous} if $\varphi^{-1}\left(R\right)\in\mathcal{Q}$, for every $R\in\mathcal{R}$. In particular, a function $f:X\to\left[-\8,+\8\right]$ is \emph{lower $\mathcal{Q}$-continuous} if it is a continuous map from $\left(X,\mathcal{Q}\right)$ into $\left(\left[-\8,+\8\right],\mathcal{R}\right)$, where $\mathcal{R}$ is the hull structure of lower sets. It is easy to see that $f$ is lower $\mathcal{Q}$-continuous iff $f^{-1}\left[-\8,r\right]\in \mathcal{Q}$, for every $r\in\R$. The collection of lower $\mathcal{Q}$-continuous functions is stable with respect to pointwise supremum, and non-negative scalar multiplication. If $f$ is lower $\mathcal{Q}$-continuous, then $f\wedge \1$ and $f+\1$ are lower $\mathcal{Q}$-continuous. Every function is lower $\Po\left(X\right)$-continuous.\medskip

We say that a hull structure $\mathcal{R}$ is \emph{$1$-algebraic}, if it is also closed with respect to arbitrary unions. It is easy to check that this property is equivalent to the fact that $A^{\mathcal{Q}}=\bigcup\limits_{x\in A}\left\{x\right\}^{\mathcal{Q}}$ (see also \cite[Theorem 7.14]{dp} for a similar result, which also justifies the term ``$1$-algebraic''). Apart from the hull operator it also generates the \emph{core operator} from $\Po\left(X\right)$ into $\mathcal{Q}$ defined by $A_{\mathcal{R}}:=\bigcup\left\{R\in\mathcal{R},~R\subset A\right\}$. Examples include $\Po\left(X\right)$ itself, balanced sets in a vector space, lower sets in a poset, solid sets in a vector lattice (see Example \ref{solid}) etc. Note that $\mathcal{R}$ generates a pre-order relation $\le_{\mathcal{R}}$ defined by $y\le_{\mathcal{R}}x$ if $y\in \left\{x\right\}^{\mathcal{R}}$; then $\mathcal{R}$ is precisely the hull structure of the lower sets with respect to this pre-order.

If $\mathcal{R}$ is $1$-algebraic, then $f$ is lower $\mathcal{R}$-continuous iff $y\in \left\{x\right\}^{\mathcal{R}}$ $\Rightarrow$ $f\left(y\right)\le f\left(x\right)$. Any function can be transformed into a lower $\mathcal{R}$-continuous one. Namely, for $f:X\to\left[-\8,+\8\right]$ define $f_{\mathcal{R}}\left(x\right):=\bigvee\limits_{y\in \left\{x\right\}^{\mathcal{R}}}f\left(y\right)$. Then, $f^{-1}_{\mathcal{R}}\left[-\8,r\right]=\left(f^{-1}\left[-\8,r\right]\right)_{\mathcal{R}}$, for every $r\in\R$.\medskip

In the following we select a ``base point'' $0_{X}\in X$. For a topology $\tau$ on $X$ let $\tau_{0}$ stand for the filter of $\tau$-neighborhoods of $0_{X}$. We call $\tau$ a $\mathcal{Q}$\emph{-topology} if it has a base at $0_{X}$ consisting of members of $\mathcal{Q}$, i.e. $\tau_{0}\cap \mathcal{Q}$ is a base for $\tau_{0}$.

\begin{proposition}\label{1ac}
Assume that $\mathcal{R}$ is a $1$-algebraic hull structure on $X$. A topology $\tau$ on $X$ is a $\mathcal{R}$-topology if and only if $x_{p}\xrightarrow[]{\tau}0_{X}$ $\Rightarrow$ $y_{p}\xrightarrow[]{\tau}0_{X}$, where $y_{p}\in\left\{x_{p}\right\}^{\mathcal{R}}$, for every $p\in P$.
\end{proposition}
\begin{proof}
Necessity: Let $U\in\tau_{0}\cap \mathcal{R}$. There is $p_{0}\in P$ such that $x_{p}\in U$, for every $p\ge p_{0}$; as $U\in\mathcal{R}$, it follows that $y_{p}\in \left\{x_{p}\right\}^{\mathcal{R}}\subset U$. Since $U$ was chosen arbitrarily we conclude that $y_{p}\xrightarrow[]{\tau}0_{X}$.\medskip

Sufficiency: Assume that $\tau$ is not a $\mathcal{R}$-topology so that there is $U\in\tau_{0}$ such that for every $V\in\tau_{0}$ we have $V^{\mathcal{R}}\not\subset U$. therefore, for every $V\in\tau_{0}$ there is $y_{V}\in U\backslash V^{\mathcal{R}}$; as $\mathcal{R}$ is $1$-algebraic, there is $x_{V}\in V$ with $y_{V}\in \left\{x_{V}\right\}^{\mathcal{R}}$. Then, $x_{V}\xrightarrow[V\in\tau_{0}]{\tau}0_{X}$, hence $y_{V}\xrightarrow[V\in\tau_{0}]{\tau}0_{X}$, which contradicts $y_{V}\notin U$, for every $V\in\tau_{0}$.
\end{proof}

\begin{proposition}\label{1a}
Assume that $\mathcal{R}$ is $1$-algebraic and $\tau$ is a $\mathcal{R}$-topology. Then:
\item[(i)] If $U\in\tau_{0}$, then $U_{\mathcal{R}}\in\tau_{0}$.
\item[(ii)] If $f:X\to\left[0,+\8\right]$ is $\tau$-continuous at $0_{X}$ and such that $f\left(0_{X}\right)=0$, then the same is true for $f_{\mathcal{R}}$.
\end{proposition}
\begin{proof}
(i): Since $\tau$ is a $\mathcal{R}$-topology, there is $V\in\tau_{0}\cap\mathcal{R}$ such that $V\subset U$. Then, by definition $V\subset U_{\mathcal{R}}$, which implies that $U_{\mathcal{R}}\in\tau_{0}$.\medskip

(ii): Let $\varepsilon>0$ and let $U:=f^{-1}\left[0,\varepsilon\right]$. Then, $U\in\tau_{0}$, and according to (i) $f^{-1}_{\mathcal{Q}}\left[0,\varepsilon\right]=U_{\mathcal{Q}}\in\tau_{0}$. As $\varepsilon$ was arbitrary, we conclude that $f_{\mathcal{R}}$ is $\tau$-continuous at $0_{X}$.
\end{proof}

Assume that $\mathcal{Q}$ and $\mathcal{R}$ are hull structures on $X$ such that $\mathcal{R}$ is $1$-algebraic. We say that $\mathcal{Q}$ is $\mathcal{R}$-\emph{enhancible} if $0_{X}\in Q_{\mathcal{R}}\in \mathcal{Q}$, for every $0_{X}\in Q\in\mathcal{Q}$; note that in fact $Q_{\mathcal{R}}\in \mathcal{Q}\cap \mathcal{R}$.

\begin{example}
If $\mathcal{R}\subset\mathcal{Q}$, then $\mathcal{Q}$ is $\mathcal{R}$-enhancible. For example, the hull structure of full sets on a poset with the smallest element $0_{X}$ is $\mathcal{L}$-enhancible, where $\mathcal{L}$ is the hull structure of lower sets.

Any hull structure on $X$ is $\Po\left(X\right)$-enhancible.

The hull structure of convex sets on a vector space is $\mathcal{B}$-enhancible, where $\mathcal{B}$ is the hull structure of balanced sets.

The hull structure of order closed sets on a vector lattice is $\mathcal{D}$-enhancible, where $\mathcal{D}$ is the hull structure of solid sets (see Example \ref{fatou}).
\qed\end{example}

\begin{proposition}\label{enhancible}
Assume that $\mathcal{Q}$ is $\mathcal{R}$-enhancible. Then:
\item[(i)] If $\tau$ is a $\mathcal{Q}$-topology and a $\mathcal{R}$-topology, then $\tau$ is a $\mathcal{Q}\cap\mathcal{R}$-topology.
\item[(ii)] If $\mathcal{S}$ is a $1$-algebraic hull structure on $X$ such that $\mathcal{Q}$ and $\mathcal{R}$ are $\mathcal{S}$-enhancible, then $\mathcal{Q}$ is $\mathcal{R}\cap\mathcal{S}$-enhancible and $Q_{\mathcal{R}\cap\mathcal{S}}=\left(Q_{\mathcal{R}}\right)_{\mathcal{S}}\in\mathcal{Q}$, for every $0_{X}\in Q\in\mathcal{Q}$.
\end{proposition}
\begin{proof}
(i): For every $U\in\tau_{0}$ there is $V\in\tau_{0}\cap\mathcal{Q}$ such that $V\subset U$. Then, by our assumption $V_{\mathcal{R}}\in\mathcal{Q}\cap\mathcal{R}$; by part (i) of Proposition \ref{1a} $V_{\mathcal{R}}\in\tau_{0}$, and by definition $V_{\mathcal{R}}\subset V\subset U$.\medskip

(ii): Let $0_{X}\in Q\in\mathcal{Q}$. Since $\mathcal{Q}$ is $\mathcal{R}$-enhancible, we have $0_{X}\in R:=Q_{\mathcal{R}}\supset Q_{\mathcal{R}\cap\mathcal{S}}$. Since $R\in\mathcal{Q}\cap \mathcal{R}$ and both $\mathcal{Q}$ and $\mathcal{R}$ are $\mathcal{S}$-enhancible, we have $0_{X}\in R_{\mathcal{S}}\in\mathcal{Q}\cap\mathcal{R}\cap \mathcal{S}$. On the other hand, $Q_{\mathcal{R}\cap\mathcal{S}}\subset R$ implies $Q_{\mathcal{R}\cap\mathcal{S}}\subset R_{\mathcal{R}\cap\mathcal{S}}\subset R_{\mathcal{S}}$, and so $Q_{\mathcal{R}\cap\mathcal{S}}=R_{\mathcal{S}}$.
\end{proof}

\begin{remark}\label{hn}
A situation similar to part (i) occurs when $\mathcal{Q}$ and $\mathcal{R}$ are such that $Q^{\mathcal{R}}\in \mathcal{Q}$, for every $0_{X}\in Q\in\mathcal{Q}$ (we do not require $\mathcal{R}$ to be $1$-algebraic). If in this case $\tau$ is a $\mathcal{Q}$-topology and a $\mathcal{R}$-topology, then $\tau$ is a $\mathcal{Q}\cap\mathcal{R}$-topology. Indeed, $\left\{U^{\mathcal{R}},~U\in\tau_{0}\cap\mathcal{Q}\right\}\subset \tau_{0}\cap\mathcal{Q}\cap\mathcal{R}$ is a basis for $\tau_{0}$: for every $W\in\tau_{0}$ there is $V\in \tau_{0}\cap\mathcal{R}$ and $U\in \tau_{0}\cap\mathcal{Q}$ such that $U\subset V\subset W$ implying $U^{\mathcal{R}}\subset V\subset W$. The hull structures in relationships like this are solid and convex sets in a vector lattice, convex and closed sets in a topological vector space etc.
\qed\end{remark}

\section{Additive topologies on monoids}\label{monoid}

Let $\left(G,+,0_{G}\right)$ be a commutative monoid, i.e. a semi-group with a unit. A \emph{pseudo-norm} on $G$ is a subadditive function $\rho:E\to \left[0,+\8\right]$ such that $\rho\left(0_{G}\right)=0$. Clearly, $\ker\rho$ and $\rho^{-1}\left(\R\right)$ are submonoids of $G$. A sum of two pseudo-norms is a pseudo-norm.

Let $\mathcal{Q}$ be a hull structure on $G$. We say that a pseudo-norm $\rho$ on $G$ is a \emph{$\mathcal{Q}$-pseudo-norm} if it is lower $\mathcal{Q}$-continuous, i.e. $\rho^{-1}\left[0,r\right]\in \mathcal{Q}$, for every $r\ge 0$. A pointwise supremum of any collection of $\mathcal{Q}$-pseudo-norms is a $\mathcal{Q}$-pseudo-norm; if $\rho$ is a $\mathcal{Q}$-pseudo-norm, then so is $\rho\wedge \1$, as well as $s\rho$, where $s\ge 0$.\medskip

A \emph{string} is a sequence $\left\{U_{n}\right\}_{n\in\N_{0}}$ of sets containing $0_{G}$ such that $U_{n}+U_{n}\subset U_{n-1}$, for every $n\in\N$. We say that a string $\left\{U_{n}\right\}_{n\in\N_{0}}$ is \emph{subordinate} to the string $\left\{V_{n}\right\}_{n\in\N_{0}}$, if $U_{n}\subset V_{n}$, for every $n\in\N_{0}$. If $\left\{U_{n}\right\}_{n\in\N_{0}}\subset\mathcal{Q}$, then we call it a $\mathcal{Q}$-\emph{string}. If $\rho$ is a pseudo-norm on $G$, then $\left\{\rho^{-1}\left[0,\frac{1}{2^{n}}\right]\right\}_{n\in\N_{0}}$ is a string; it is a $\mathcal{Q}$-string if $\rho$ is a $\mathcal{Q}$-pseudo-norm.

We call $\mathcal{Q}$ \emph{additive} if $Q,R\in\mathcal{Q}$ $\Rightarrow$ $Q+R\in\mathcal{Q}$. For $e\in G$ let $T_{e}:G\to G$ be the \emph{translation map} defined by $T_{e}g:=e+g$. We say that $\mathcal{Q}$ is \emph{translation-invariant} if every $T_{e}$ is a $\mathcal{Q}$-continuous map on $G$.

As was mentioned before, every pseudo-norm gives rise to a string. We will now show how a string gives rise to a pseudo-norm. The arguments for the first case are adapted from \cite{aek} and \cite{wael}; the idea for the proof of second case is inspired by \cite{fremlin}.

\begin{theorem}\label{strps}
Assume that $\mathcal{Q}$ is either additive or translation invariant and $\left\{U_{n}\right\}_{n\in\N_{0}}$ is a $\mathcal{Q}$-string in $G$. Then there is a $\left[0,1\right]$-valued $\mathcal{Q}$-pseudo-norm $\rho$ on $G$ such that $\rho^{-1}\left[0,\frac{1}{2^{n}}\right)\subset U_{n}\subset \rho^{-1}\left[0,\frac{1}{2^{n}}\right]$, for every $n\in\N_{0}$, and in particular $\ker\rho=\bigcap\limits_{n\in\N_{0}}U_{n}$.
\end{theorem}
\begin{proof}
\textbf{Additive case.} In this proof if $n\in\N$ and $A\subset G$, the notation $nA$ means $A+...+A$ ($n$ times).

Let $\Q_{2}$ denote the set of positive dyadic rationals. For $q\in\Q_{2}$ let $V_{q}:=\sum\limits_{i=0}^{n}k_{i}U_{i}$, where $k_{0}\in\N_{0}$ and $k_{1},...,k_{n}\in\left\{0,1\right\}$ are such that $q=\sum\limits_{i=0}^{n}k_{i}2^{-i}$. Clearly, $V_{q}\in\mathcal{Q}$. Note that $V_{2^{-n}}=U_{n}$, for every $n\in\N_{0}$.

Let us show that $V_{p}+V_{q}\subset V_{p+q}$, for any $p,q\in\Q_{2}$ (and in particular, if $q\le r$, then $V_{q}\subset V_{q}+V_{r-q}\subset V_{r}$). Namely, we will prove by induction on $n$ that if $p=\sum\limits_{i=0}^{n}k_{i}2^{-i}$ and $q=\sum\limits_{i=0}^{n}l_{i}2^{-i}$, then $V_{p}+V_{q}\subset V_{p+q}$. The case $n=0$ is obvious. Assume that the claim is proven for $n-1$ and let $p'=\sum\limits_{i=0}^{n-1}k_{i}2^{-i}$, $q'=\sum\limits_{i=0}^{n-1}l_{i}2^{-i}$. By the assumption of induction we have $V_{p'}+V_{q'}\subset V_{p'+q'}$.

If $k_{n}=1$, and $l_{n}=0$, we have $V_{p}=V_{p'}+U_{n}$, $V_{q}=V_{q'}$ and $V_{p+q}=V_{p'+q'}+U_{n}\supset V_{p'}+V_{q'}+U_{n}=V_{p}+V_{q}$. The case of $k_{n}=0$, and $l_{n}=1$ is similar, and the case of $k_{n}=l_{n}=0$ is trivial. If $k_{n}=1=l_{n}$, then $V_{p}=V_{p'}+U_{n}$, $V_{q}=V_{q'}+U_{n}$, and $$V_{p}+V_{q}=V_{p'}+U_{n}+V_{q'}+U_{n}\subset V_{p'+q'}+U_{n-1}= V_{p'+q'}+V_{2^{1-n}}\subset V_{p'+q'+2^{1-n}}=V_{p+q},$$ where the last inclusion follows from the assumption of induction.\medskip

Define $\rho:E\to \left[0,\8\right]$ by $\rho\left(e\right):=\bigwedge\left\{q\in\Q_{2},~ e\in V_{q}\right\}$. Using the fact that $\left(V_{q}\right)_{q\in\Q_{2}}$ is increasing, one can show that $\rho^{-1}\left[0,r\right]=\bigcap\limits_{\Q_{2}\ni q> r}V_{q}\in\mathcal{Q}$ and $\rho^{-1}\left[0,r\right)=\bigcup\limits_{Q_{2}\ni q< r}V_{q}$, for every $r>0$. Note that $\rho^{-1}\left[0,q\right)\subset V_{q}\subset \rho^{-1}\left[0,q\right]$, for every $q\in\Q_{2}$, and in particular $\rho^{-1}\left[0,2^{-n}\right)\subset U_{n}\subset \rho^{-1}\left[0,2^{-n}\right]$, for every $n\in\N_{0}$. It follows that $\ker \rho=\bigcap\limits_{n\in\N}U_{n}\ni 0_{G}$.

For $f,g\in G$, for every $\Q_{2}\ni p>\rho\left(f\right)$ and $\Q_{2}\ni q>\rho\left(g\right)$ we have that $f\in V_{p}$ and $g\in V_{q}$, therefore $f+g\in V_{p+q}$, hence $\rho\left(f+g\right)\le p+q$. We conclude that $\rho$ is subadditive. Finally, $\rho\wedge\1$ is the desired $\left[0,1\right]$-valued $\mathcal{Q}$-pseudo-norm.\medskip

\textbf{Translation invariant case.} First, note that $\Po\left(G\right)$ is an additive hull structure, and so by the previous case there is a $\left[0,1\right]$-valued pseudo-norm $\rho$ on $G$ such that $\rho^{-1}\left[0,\frac{1}{2^{n}}\right)\subset U_{n}\subset \rho^{-1}\left[0,\frac{1}{2^{n}}\right]$, for every $n\in\N_{0}$. Let $\Lambda$ be the set of all lower $\mathcal{Q}$-continuous functions $\lambda\le \rho$. Let $\theta:=\bigvee\Lambda\le \rho$, which is lower $\mathcal{Q}$-continuous, and $\theta\le \rho\le\1$. For every $n\in\N_{0}$ we have $\frac{1}{2^{n}}\1_{G\backslash U_{n}}\in\Lambda$, hence $\frac{1}{2^{n}}\1_{G\backslash U_{n}}\le\theta$, and so  $\theta^{-1}\left[0,\frac{1}{2^{n}}\right)\subset U_{n}$; on the other hand $\theta\le\rho$ yields $U_{n}\subset \rho^{-1}\left[0,\frac{1}{2^{n}}\right]\subset \theta^{-1}\left[0,\frac{1}{2^{n}}\right]$.

It is left show that $\theta$ is subadditive. Fix $e\in G$ and consider $\theta_{e}:=\theta\left(\cdot+e\right)-\rho\left(e\right)$. For every $g\in G$ we have $\theta_{e}\left(g\right):=\theta\left(g+e\right)-\rho\left(e\right)\le \rho\left(g+e\right)-\rho\left(e\right)\le \rho\left(g\right)$, and so $\theta_{e}\le\rho$. We claim that $\theta_{e}$ is lower $\mathcal{Q}$-continuous. Indeed,  $\theta_{e}\left(g\right)\le r$ $\Leftrightarrow$ $\theta\left(g+e\right)\le \rho\left(e\right)+r$ $\Leftrightarrow$ $T_{e}g\in \theta^{-1}\left(-\8,\rho\left(e\right)+r\right]\in\mathcal{Q}$, and so $\theta^{-1}_{e}\left(-\8,r\right]=T_{e}^{-1}\theta^{-1}\left(-\8,\rho\left(e\right)+r\right]\in \mathcal{Q}$. It follows that $\theta_{e}\in\Lambda$, and so $\theta_{e}\le \theta$, from where $\theta\left(g+e\right)-\rho\left(e\right)\le \theta\left(g\right)$, for every $g\in G$. Next, we fix $f\in G$ and consider $\lambda_{f}:=\theta\left(\cdot+f\right)-\theta\left(f\right)$. Analogously to the previous step, we show that $\lambda_{f}\in\Lambda$, so that $\lambda_{f}\le \theta$, and so $\theta\left(f+g\right)-\theta\left(f\right)\le \theta\left(g\right)$, for every $g\in G$.
\end{proof}

A topology $\tau$ on $G$ is \emph{additive} if the addition is continuous from $\left(G,\tau\right)\times \left(G,\tau\right)$ into $\left(G,\tau\right)$. We denote the filter of $\tau$-neighborhoods of $0_{G}$ by $\tau_{0}$. The following is standard.

\begin{proposition}\label{string0}
Assume that $\tau$ is an additive topology on $G$. If $U\in\tau_{0}$, there is $V\in\tau_{0}$ such that $V+V\subset U$.
\end{proposition}
\begin{proof}
As $\bigplus:G\times G\to G$ is continuous, there are $V_{1},V_{2}\in\tau_{0}$ such that $\bigplus\left(V_{1}\times V_{2}\right)\subset U$. Then $V:=V_{1}\cap V_{2}$ does the job.
\end{proof}

We call $\tau$ a \emph{$\mathcal{Q}$-topology} if it has a base at $0_{G}$ consisting of members of $\mathcal{Q}$, i.e. $\tau_{0}\cap \mathcal{Q}$ is a base for $\tau_{0}$.

\begin{lemma}\label{string}
Assume that $\tau$ is an additive $\mathcal{Q}$-topology on $G$. For every sequence $\left\{W_{n}\right\}_{n\in\N_{0}}\subset\tau_{0}$  there is a $\mathcal{Q}$-string $\left\{U_{n}\right\}_{n\in\N_{0}}\subset\tau_{0}$ such that $U_{n}\subset W_{n}$, for every $n\in\N_{0}$
\end{lemma}
\begin{proof}
We construct the string as well as an auxiliary sequence $\left\{V_{n}\right\}_{n\in\N_{0}}\subset\tau_{0}$ inductively. Let $V_{0}:=G$. If $V_{n}$ is constructed, since $\tau_{0}\cap \mathcal{Q}$ is a base for $\tau_{0}$, there is $U_{n}\in\tau_{0}\cap\mathcal{Q}$ such that $U_{n}\subset V_{n}\cap W_{n}$. According to Proposition \ref{string0} there is $V_{n+1}\in\tau_{0}$ such that $V_{n+1}+V_{n+1}\subset U_{n}$. This way we always have $U_{n}\subset W_{n}$ and $U_{n+1}+U_{n+1}\subset V_{n+1}+V_{n+1}\subset U_{n}$, for every $n\in\N_{0}$.
\end{proof}

We now add the second hull structure into the picture. We call a hull structure $\mathcal{R}$ on $G$ \emph{basic} if it is $1$-algebraic additive and such that $\left\{0_{G}\right\}\in\mathcal{R}$.

\begin{proposition}\label{1ap}
Let $\mathcal{R}$ be a basic hull structure on $G$ and let $\rho$ be a pseudo-norm on $G$. Then, $\rho_{\mathcal{R}}$ is a $\mathcal{R}$-pseudo-norm.
\end{proposition}
\begin{proof}
Lower $\mathcal{R}$-continuity is true for any function, and so we only need to show that $\rho_{\mathcal{R}}$ is a pseudo-norm. First, since $\left\{0_{G}\right\}\in\mathcal{R}$, it follows that $\left\{0_{G}\right\}^{\mathcal{R}}=\left\{0_{G}\right\}$, and so $\rho_{\mathcal{R}}\left(0_{G}\right)=\rho\left(0_{G}\right)=0$. Let $f,g\in G$. Additivity of $\mathcal{R}$ implies that $f+g\in\left\{f\right\}^{\mathcal{R}}+\left\{g\right\}^{\mathcal{R}}\in\mathcal{R}$, and so $\left\{f+g\right\}^{\mathcal{R}}\subset \left\{f\right\}^{\mathcal{R}}+\left\{g\right\}^{\mathcal{R}}$. For every $h\in\left\{f+g\right\}^{\mathcal{R}}$ there are $u\in \left\{f\right\}^{\mathcal{R}}$ and $v\in\left\{g\right\}^{\mathcal{R}}$ such that $h=u+v$. Therefore, $\rho\left(h\right)=\rho\left(u+v\right)\le\rho\left(u\right)+\rho\left(v\right)\le \rho_{\mathcal{R}}\left(f\right)+\rho_{\mathcal{R}}\left(g\right)$. Taking supremum over all $h$ yields $\rho_{\mathcal{R}}\left(f+g\right)\le \rho_{\mathcal{R}}\left(f\right)+\rho_{\mathcal{R}}\left(g\right)$.
\end{proof}

We can now state the main operating result of the section.

\begin{theorem}\label{pseudo}
Assume that we are given the following data:
\begin{itemize}
\item Hull structures $\mathcal{Q}$ and $\mathcal{R}$ on $G$ such that $\mathcal{R}$ is basic, and $\mathcal{Q}$ is $\mathcal{R}$-enhancible and either additive or translation invariant;
\item An additive topology $\tau$ on $G$ which is a $\mathcal{Q}$-topology and a $\mathcal{R}$-topology.
\end{itemize}
Then, for every sequence $\left\{W_{n}\right\}_{n\in\N_{0}}\subset\tau_{0}$ there is a $\left[0,1\right]$-valued $\mathcal{Q}\cap\mathcal{R}$-pseudo-norm $\rho$, which  is $\tau$-continuous at $0_{G}$ and such that $\rho^{-1}\left[0,\frac{1}{2^{n}}\right)\subset W_{n}$, for every $n\in\N_{0}$.
\end{theorem}
\begin{proof}
First, using Lemma \ref{string} construct a $\mathcal{Q}$-string $\left\{U_{n}\right\}_{n\in\N_{0}}\subset\tau_{0}$, subordinate to $\left\{W_{n}\right\}_{n\in\N_{0}}$. Then, using Theorem \ref{strps} find a $\left[0,1\right]$-valued $\mathcal{Q}$-pseudo-norm $\rho$ with $\rho^{-1}\left[0,\frac{1}{2^{n}}\right)\subset U_{n}\subset \rho^{-1}\left[0,\frac{1}{2^{n}}\right]$, for every $n\in\N_{0}$. The first inclusion implies that $\rho^{-1}\left[0,1\right)\subset W_{n}$, for every $n\in\N_{0}$, while the second implies that $\rho$ is $\tau$-continuous at $0_{G}$. By Proposition \ref{1ap} and part (ii) of Proposition \ref{1a} $\rho_{\mathcal{R}}$ is a $\left[0,1\right]$-valued pseudo-norm, which is $\tau$-continuous at $0_{G}$. Moreover,  $\rho_{\mathcal{R}}$ is a $\mathcal{Q}\cap\mathcal{R}$-pseudo-norm, as $\rho^{-1}_{\mathcal{R}}\left[0,r\right]=\left(\rho^{-1}\left[0,r\right]\right)_{\mathcal{R}}\in\mathcal{Q}\cap\mathcal{R}$, for every $r\ge 0$. Finally, $\rho_{\mathcal{R}}\ge\rho$ yields $\rho^{-1}_{\mathcal{R}}\left[0,\frac{1}{2^{n}}\right)\subset\rho^{-1}\left[0,\frac{1}{2^{n}}\right)\subset W_{n}$, for every $n\in\N_{0}$.
\end{proof}

\section{Group topologies}\label{group}

Let $G$ be a commutative group. If $\tau$ is an additive topology on $G$, then for every $e\in G$ the translation map $T_{e}$ is a homeomorphism. We call $\tau$ a \emph{group topology} if taking inverse is a continuous operation (in fact a homeomorphism) on $G$. A set $A\subset G$ is called \emph{symmetric} if $A=-A$. Let $\mathcal{S}$ be the hull structure of symmetric sets. A hull structure $\mathcal{Q}$ on $G$ is \emph{symmetric} if $Q\in\mathcal{Q}$ $\Rightarrow$ $-Q\in\mathcal{Q}$. Note that in this case $Q_{\mathcal{S}}=Q\cap -Q\in\mathcal{Q}$.

\begin{proposition}\label{sym}
Assume that $\mathcal{Q}$ is a symmetric hull structure on $G$. For an additive $\mathcal{Q}$-topology $\tau$ on $G$ the following conditions are equivalent:
\item[(i)] $\tau$ is a group topology;
\item[(ii)] If $g_{p}\xrightarrow[]{\tau}0_{G}$, then $-g_{p}\xrightarrow[]{\tau}0_{G}$;
\item[(iii)] $\tau$ is a $\mathcal{S}$-topology;
\item[(iv)] $\tau$ is a $\mathcal{Q}\cap \mathcal{S}$-topology.
\end{proposition}
\begin{proof}
(i)$\Rightarrow$(ii) is obvious; the converse follows from the standard arguments and the fact that the translations are homeomorphisms. (ii)$\Leftrightarrow$(iii) can be deduced from Proposition \ref{1ac}, (iii)$\Leftrightarrow$(iv) follows from part (iii) of Proposition \ref{1a}.
\end{proof}

We call a pseudo-norm $\rho$ on $G$ \emph{symmetric} if $\rho\left(-g\right)=\rho\left(g\right)$, for every $g\in G$, which is equivalent to lower $\mathcal{S}$-continuity of $\rho$. Clearly, $\ker\rho$ and $\rho^{-1}\left(\R\right)$ are subgroups of $G$. If $\rho$ is a pseudo-norm, $\rho_{-1}\left(g\right):=\rho\left(-g\right)$ defines a pseudo-norm, and $\rho_{\mathcal{S}}=\rho\vee\rho_{-1}$ is a symmetric pseudo-norm. If $\tau$ is a group topology, $\mathcal{Q}$ is a symmetric hull structure, and $\rho$ is a $\tau$-continuous $\mathcal{Q}$-pseudo-norm, then $\rho_{-1}$ and $\rho_{\mathcal{S}}$ are $\tau$-continuous $\mathcal{Q}$-pseudo-norms.

A symmetric pseudo-norm $\rho$ generates a pseudo-metric on $G$ via $d_{\rho}:=\rho\left(\cdot-\cdot\right)$. Then $d_{\rho}$ is translation invariant, i.e. every translation map is a $d_{\rho}$-isometry. Conversely, every translation invariant pseudo-metric $d:G\times G\to \left[0,+\8\right]$ generates a symmetric pseudo-norm on $G$ by $\rho_{d}:=d\left(\cdot,0_{G}\right)$. It is easy to see that the two operations are inverses of each other and that the topology defined by $d_{\rho}$ is a group topology.

\begin{proposition}\label{psco}
For an additive topology $\tau$ on $G$ and a symmetric pseudo-norm $\rho$ the following conditions are equivalent:
\item[(i)] $\tau$ is stronger than the topology generated by $d_{\rho}$;
\item[(ii)] $d_{\rho}$ is continuous on $G\times G$;
\item[(iii)] $\rho$ continuous on $G$;
\item[(iv)] $\rho$ continuous at $0_{G}$;
\item[(v)] Every $\rho$-ball at $0_{G}$ is a $\tau$-neighborhood of $0_{G}$.
\end{proposition}
\begin{proof}
(i)$\Rightarrow$(ii) is true for any pseudo-metric, (ii)$\Rightarrow$(iii) and (iv)$\Rightarrow$(v)$\Rightarrow$(i) are straightforward, (iii)$\Rightarrow$(iv) is trivial.
\end{proof}

In particular, the topology generated by $d_{\rho}$ is the weakest additive topology in which $\rho$ is continuous at $0_{G}$. Any collection of symmetric pseudo-norms generates a group topology in this sense. The convergence in this topology can be described by $g_{p}\to g$ iff $\rho\left(g_{p}-g\right)\to 0$, for each $\rho$ in the collection. If the involved pseudo-norms are $\mathcal{Q}$-pseudo-norms, the resulting topology is a $\mathcal{Q}$-topology.\medskip

For two symmetric pseudo-norms $\rho,\lambda$, we say that $\lambda$ is \emph{$\rho$-continuous}, if the topology generated by $\lambda$ is weaker than that generated by $\rho$. If also $\rho$ is $\lambda$-continuous, we say that they are \emph{equivalent}.

\begin{proposition}\label{pscont}For symmetric pseudo-norms $\rho,\lambda$ the following conditions are equivalent:
\item[(i)] $\lambda$ is $\rho$-continuous;
\item[(ii)] Whenever $g_{p}\xrightarrow[]{\rho}0_{G}$, it follows that $g_{p}\xrightarrow[]{\lambda}0_{G}$;
\item[(iii)] Whenever $g_{n}\xrightarrow[]{\rho}0_{G}$, it follows that $g_{n}\xrightarrow[]{\lambda}0_{G}$;
\item[(iv)] For every $\varepsilon>0$ there is $\delta>0$ such that whenever $\rho\left(g\right)<\delta$, it follows that $\lambda\left(g\right)<\varepsilon$.
\end{proposition}
\begin{proof}
(iv)$\Rightarrow$(i) is routine, while (i)$\Rightarrow$(ii)$\Rightarrow$(iii) are trivial.

(iii)$\Rightarrow$(iv): Negation of (iv) means that there are $\varepsilon>0$ and $\left(g_{n}\right)_{n\in\N}\subset G$ such that $\rho\left(g_{n}\right)<\frac{1}{n}$ and $\lambda\left(g_{n}\right)\ge\varepsilon$, for every $n\in\N$, which contradicts (iii).
\end{proof}

In particular, $\rho$ and $\rho\wedge\1$ are equivalent. It also follows that if $\lambda\le\rho$, then $\lambda$ is $\rho$-continuous. Hence, the set of all pseudo-norms which are continuous with respect to an additive topology $\tau$ is a lower set in the collection of all symmetric pseudo-norms on $G$. We now show that a sequence of pseudo-norms generate the same topology as a single one.

\begin{lemma}\label{sp}
Let $\tau$ be an additive topology on $G$. If $\left\{\rho_{n}\right\}_{n\in\N}$ is a sequence of symmetric $\tau$-continuous $\mathcal{Q}$-pseudo-norms on $G$, then $\rho:=\bigvee\limits_{n\in\N}\frac{\rho_{n}\wedge\1}{2^{n}}$ (pointwise supremum) is a symmetric $\tau$-continuous $\mathcal{Q}$-pseudo-norm such that $\rho_{n}$ is $\rho$-continuous, for every $n\in\N$.
\end{lemma}
\begin{proof}
The only non-obvious claim is $\tau$-continuity of $\rho$. Since $\rho\le\sum\limits_{n\in\N}\frac{\rho_{n}\wedge\1}{2^{n}}$, which is a $\tau$-continuous pseudo-norm, it follows that $\rho$ is $\tau$-continuous.
\end{proof}

We can now derive a number of corollaries from Theorem \ref{pseudo}. Part (iv) in the following result is motivated by \cite[Section 3]{ddje}.

\begin{theorem}\label{generated}
Assume that we are given the following data:
\begin{itemize}
\item Symmetric hull structures $\mathcal{Q}$ and $\mathcal{R}$ on $G$ such that $\mathcal{R}$ is basic, and $\mathcal{Q}$ is $\mathcal{R}$-enhancible and either additive or translation invariant;
\item A group topology $\tau$ on $G$ which is a $\mathcal{Q}$-topology and a $\mathcal{R}$-topology.
\end{itemize}
Then:
\item[(i)] $\tau$ is generated by a collection of symmetric $\tau$-continuous $\left[0,1\right]$-valued $\mathcal{Q}\cap\mathcal{R}$-pseudo-norms.
\item[(ii)] If $\tau_{0}$ has a countable base, then $\tau$ is generated by a single symmetric $\tau$-continuous $\left[0,1\right]$-valued $\mathcal{Q}\cap\mathcal{R}$-pseudo-norm.
\item[(iii)] Assume that $f:G\to\R$ with $f\left(0_{G}\right)=0$ is continuous at $0_{G}$. Then, there is a symmetric $\tau$-continuous $\mathcal{Q}\cap\mathcal{R}$-pseudo-norm $\rho$ such that $f$ is $\rho$-continuous at $0_{G}$. In particular, $\ker\rho\subset\ker f$.
\item[(iv)] If $\tau$ is submetrizable (i.e. there is a weaker metrizable topology), then there is a total (a.k.a. non-degenerate) symmetric $\tau$-continuous $\mathcal{Q}\cap\mathcal{R}$-pseudo-norm.
\item[(v)] Assume that $\lambda$ is a $\tau$-continuous pseudo-norm on $G$. Then, there is a symmetric $\tau$-continuous $\left[0,1\right]$-valued $\mathcal{Q}\cap\mathcal{R}$-pseudo-norm $\rho$ on $G$, such that $\lambda$ is $\rho$-continuous.
\end{theorem}
\begin{proof}
First, it is easy to see that $\mathcal{R}\cap\mathcal{S}$ is a basic hull structure; moreover by part (ii) of Proposition \ref{enhancible}, $\mathcal{Q}$ is $\mathcal{R}\cap\mathcal{S}$-enhancible. By Proposition \ref{sym} $\tau$ is a $\mathcal{R}\cap\mathcal{S}$ topology.

(i): Take $W\in\tau_{0}$; according to Theorem \ref{pseudo} we can find a $\left[0,1\right]$-valued $\mathcal{Q}\cap\mathcal{R}\cap\mathcal{S}$-pseudo-norm $\rho$, which is $\tau$-continuous at $0_{G}$ and such that $\rho^{-1}\left[0,1\right)\subset W$. Clearly, $\rho$ is symmetric, and by Proposition \ref{psco}, $\rho$ is $\tau$-continuous. As $W$ was arbitrary, the claim follows.

(ii) is proven similarly, but starting with a basis $\left\{W_{n}\right\}_{n\in\N_{0}}$ for $\tau_{0}$.

(iii): We may assume that $f\left(0\right)=0$. Continuity of $f$ at $0_{G}$ implies that $W_{n}:=f^{-1}\left(-\frac{1}{n},\frac{1}{n}\right)\in\tau_{0}$, for every $n\in\N$. After this proceed as in (i) and (ii).

(iv) follows from (iii) for $f:=d\left(\cdot,0_{G}\right)$, where $d$ is the metric for a topology weaker than $\tau$. (v) follows from (iii) and Proposition \ref{psco}.
\end{proof}

\begin{remark}
Note that it is required in the definition of a hull structure that $G\in\mathcal{Q}$. However, if $\mathcal{Q}$ is closed with respect to arbitrary intersections, but contains no $G$, the latter can just be added to $\mathcal{Q}$ for it to become a hull structure $\mathcal{Q}'$. Applying Theorem \ref{pseudo} to $\mathcal{Q}'$ we get a $\mathcal{Q}'$-pseudo-norm $\rho$ which is not necessarily lower $\mathcal{Q}$ continuous, but such that small enough $\rho$-balls are members of $\mathcal{Q}$.

A possible situation when this idea is applicable is when $\mathcal{Q}$ is the collection of sets which are complete in a certain sense, but $G$ is too large to be complete (such as for example bounded uo completeness of a Banach lattice, see \cite[Theorem 2.6]{taylor}).
\qed\end{remark}

\begin{remark}
Note that it is possible to obtain analogues of the result of this section at the level of monoids. However, we cannot generate a pseudo-metric from a pseudo-norm, due to lack of subtraction. Hence, a pseudo-norm does not generate a topology on a monoid, instead it generates a pre-topology at $0_{G}$, and so instead of working with topologies, the ``correct'' setting for monoids are additive pre-topologies at $0_{G}$. For more details see \cite{erzh}.
\qed\end{remark}

\section{Linear topologies}\label{vs}

Let $E$ be a vector space over $\R$. A set $A\subset G$ is called \emph{balanced} if $rA\subset A$, for $\left|r\right|\le 1$. Every balanced set is symmetric. We call a hull structure $\mathcal{Q}$ on $G$ \emph{contraction invariant} if the map of the form $e\mapsto re$ is $\mathcal{Q}$-continuous, for every $\left|r\right|\le 1$. This is equivalent to the fact that $Q\in\mathcal{Q}$, $\left|r\right|\ge 1$ $\Rightarrow$ $rQ\in\mathcal{Q}$ implying that $\mathcal{Q}$ is symmetric. Moreover, if $\mathcal{B}$ is the hull structure of balanced sets, and $\mathcal{Q}$ is contraction invariant, then $Q_{\mathcal{B}}=\bigcap\limits_{\left|r\right|\ge 1}rQ\in\mathcal{Q}$, for every $0_{E}\in Q\in\mathcal{Q}$. Thus, $\mathcal{Q}$ is $\mathcal{B}$-enhancible.

\begin{proposition}\label{balt}
Assume that $\mathcal{Q}$ is a contraction invariant hull structure on $G$. For an additive $\mathcal{Q}$-topology $\tau$ on $E$ the following conditions are equivalent:
\item[(i)] $\tau$ is a $\mathcal{B}$-topology;
\item[(ii)] $\tau$ is a $\mathcal{Q}\cap \mathcal{B}$-topology;
\item[(iii)] If $e_{p}\xrightarrow[]{\tau}0_{E}$, and $\left|r_{p}\right|\le 1$, for every $p$, then $r_{p}e_{p}\xrightarrow[]{\tau}0_{E}$;
\item[(iv)] If $e_{p}\xrightarrow[]{\tau}0_{E}$, and $\left\{r_{p}\right\}_{p\in P}\subset \R$ is bounded, then $r_{p}e_{p}\xrightarrow[]{\tau}0_{E}$.
\end{proposition}
\begin{proof}
(i)$\Leftrightarrow$(ii) follow from part (ii) of Proposition \ref{1a}. (i)$\Leftrightarrow$(iii) follow from the comment before the proposition and Proposition \ref{1ac}.

(iv)$\Rightarrow$(iii) is trivial. For the converse find $n\in\N$ be such that $\left|r_{p}\right|\le n$, for every $n\in\N$; we have $\frac{r_{p}}{n}e_{p}\xrightarrow[]{\tau}0_{E}$, hence $r_{p}e_{p}=\frac{r_{p}}{n}e_{p}+...+\frac{r_{p}}{n}e_{p}\xrightarrow[]{\tau}0_{E}$.
\end{proof}

We will call $\tau$ \emph{Archimedean} if $\frac{1}{n}e\xrightarrow[]{\tau} 0_{E}$, for every $e\in E$, and \emph{linear} if it is an additive topology such the scalar multiplication is a continuous map from $\R\times E$ into $E$. Clearly, every linear topology is an Archimedean group topology.

\begin{proposition}\label{arb}
A topology $\tau$ on $E$ is linear if and only if it is an additive Archimedean $\mathcal{B}$-topology.
\end{proposition}
\begin{proof}
Necessity: We only need to show local balanced-ness. Let $U\in\tau_{0}$. Due to continuity of the scalar multiplication that there are $V\in\tau_{0}$ and $r>0$ such that $\left(-r,r\right)V\subset U$. Since $e\mapsto \frac{r}{2}e$ is a homeomorphism, it follows that $\frac{r}{2}V\in\tau_{0}$. It is easy to see that $\left(-r,r\right)V$ is balanced; since it contains $\frac{r}{2}V$, it is a neighborhood of $0_{E}$.\medskip

Sufficiency: Let $U\in\tau_{0}$ and let $V\in\tau_{0}\cap\mathcal{B}$ be such that $V\subset U$. Fix $e\in E$. As $\tau$ is Archimedean, it follows that there is $n\in\N$ such that $\frac{1}{n}e\in V$, but since the latter is balanced we have  $\left(-\frac{1}{n},\frac{1}{n}\right)e\subset V\subset U$. Due to arbitrariness of $U$ it follows that the map $r\mapsto re$ is continuous at $0$. Next, suppose that $e_{p}\xrightarrow[]{\tau}e$ and $r_{p}\xrightarrow[]{\tau}r$. In particular, by passing to a tail we may assume that $\left\{r_{p}\right\}_{p\in P}\subset \R$ is bounded, and so according to Proposition \ref{balt} we get $r_{p}\left(e_{p}-e\right)\xrightarrow[]{\tau}0_{E}$. On the other hand, by the previous step we have $\left(r_{p}-r\right)e\xrightarrow[]{\tau}0_{E}$. We conclude that $r_{p}e_{p}-re=r_{p}\left(e_{p}-e\right)+\left(r_{p}-r\right)e\xrightarrow[]{\tau}0_{E}$.
\end{proof}

We call a pseudo-norm $\rho$ on $G$ \emph{balanced} if $\rho\left(r e\right)\le \rho\left(e\right)$, for every $e\in E$ and $\left|r\right|\le 1$. Note that $\rho$ is balanced iff it is a $\mathcal{B}$-pseudo-norm. In this case $\rho$ is symmetric, and $\ker\rho$ and $\rho^{-1}\left(\R\right)$ are subspaces of $E$. We also call $\rho$ \emph{Archimedean} if $\rho\left(\frac{1}{n} e\right)\to 0$, for every $e\in E$. If $\rho$ is Archimedean, then so is $\rho\wedge \1$. If $\rho$ is continuous with respect to a linear topology then it is Archimedean.

\begin{proposition}\label{bal}
For a balanced pseudo-norm $\rho$ on $E$ the following conditions are equivalent:
\item[(i)] $\rho$ generates a linear topology;
\item[(ii)] $\rho$ is Archimedean;
\item[(iii)] $\rho$-balls at $0_{E}$ are absorbing.
\end{proposition}
\begin{proof}
(i)$\Rightarrow$(ii)$\Leftrightarrow$(iii) are straightforward.

(ii)$\Rightarrow$(i): Since $\rho$ is balanced and Archimedean it generates an Archimedean $\mathcal{B}$-topology. Such topology is linear according to Proposition \ref{arb}.
\end{proof}

If $\rho$ is a pseudo-norm, for $\left|r\right|\le 1$ define a pseudo-norm $\rho_{r}$ by $\rho_{r}\left(e\right):=\rho\left(re\right)$. Then, $\rho_{\mathcal{B}}:=\bigvee\limits_{\left|r\right|\le 1}\rho_{r}$ is a balanced pseudo-norm. The following result is partially taken from \cite[Chapter 26]{schechter}.

\begin{proposition}For a symmetric pseudo-norm $\rho$ on $E$ the following conditions are equivalent:
\item[(i)] $\rho$ is equivalent to a balanced Archimedean pseudo-norm;
\item[(ii)] $\rho$ is a generates a linear topology;
\item[(iii)] If $\left(e_{p}\right)_{p\in P}\subset E$, $e\in E$, $\left(r_{p}\right)_{p\in P}\subset \R$ and $r\in\R$ are such that $\rho\left(e_{p}-e\right)\to 0$ and $\left|r_{p}-r\right|\to 0$, it follows that $\rho\left(r_{p}e_{p}-re\right)\to 0$;
\item[(iv)] Whenever $\rho\left(e_{n}\right)\to 0$, $r_{n}\to 0$, $e\in E$ and $r\in\R$, it follows that $\rho\left(r e_{n}\right)\to 0$ and $\rho\left(r_{n}e\right)\to 0$.
\end{proposition}
\begin{proof}
(i)$\Rightarrow$(ii) follows from Proposition \ref{bal}. (ii)$\Rightarrow$(iii) is straightforward. (iii)$\Rightarrow$(iv) is trivial.

(iv)$\Rightarrow$(i): First, let us prove that if $\rho\left(e_{n}\right)\to 0$ and $r_{n}\to 0$, then $\rho\left(r_{n} e_{n}\right)\to 0$. For every $n\in\N$ we may view $e_{n}$ as a continuous linear map from $\R$ into $E$ defined by $e_{n}\left(s\right):=se_{n}$. For every $s\in\R$ we have that $\rho\left(se_{n}\right)\to 0$, and so in particular $\left\{se_{n}\right\}_{n\in\N}$ is bounded with respect to $\rho$-topology. Hence, $\left\{e_{n}\right\}_{n\in\N}$ is a a pointwise bounded collection of continuous linear maps on a Baire topological vector space $\R$, and so is equicontinuous (see \cite[Theorem 11.9.5]{bn}). Hence, for every $\varepsilon>0$ there is $\delta>0$ such that whenever $\left|s\right|<\delta$, then $\rho\left(se_{n}\right)<\varepsilon$, for every $n\in\N$. In particular, if $r_{n}\to 0$, we have $\rho\left(r_{n}e_{n}\right)<\varepsilon$, for large enough $n$.\medskip

Next, since $\rho$ is equivalent to $\rho\wedge\1$ we may assume that $\rho$ is $\left[0,1\right]$-valued. We have that $\rho_{\mathcal{B}}\ge\rho$ is a balanced pseudo-norm. To show that $\rho_{\mathcal{B}}$ is $\rho$-continuous assume the opposite, and so according to Proposition \ref{pscont} there are $\varepsilon>0$ and a sequence $\left(e_{n}\right)_{n\in\N}$ such that $\rho\left(e_{n}\right)\to 0$, and $\rho_{\mathcal{B}}\left(e_{n}\right)>\varepsilon$, for every $n\in\N$. Then, for every $n\in\N$ there is $r_{n}\in\left[-1,1\right]$ such that $\rho\left(r_{n}e_{n}\right)>\varepsilon$. Passing to a subsequence, we may assume that $r_{n}\to r$, and so by the first step we have $\varepsilon<\rho\left(r_{n}e_{n}\right)\le \rho\left(\left(r_{n}-r\right)e_{n}\right)+\rho\left(re_{n}\right)\to 0$. Contradiction.
\end{proof}

If $\mathcal{Q}$ is a contraction invariant hull structure, and $\rho$ is a $\mathcal{Q}$-pseudo-norm, then $\rho_{r}$ a $\mathcal{Q}$-pseudo-norms, for every $r\in\left[-1,1\right]$, which implies that $\rho_{\mathcal{B}}$ is a $\mathcal{B}\cap\mathcal{Q}$-pseudo-norm. Moreover, it follows from part (ii) of Proposition \ref{1a} and Proposition \ref{psco} that if
$\tau$ is a group $\mathcal{B}$-topology and $\rho$ is a $\tau$-continuous pseudo-norm, then $\rho_{\mathcal{B}}$ is also $\tau$-continuous. We can now establish the analogue of Theorem \ref{generated} for linear topologies (we only state the analogue of part (i); the formulations and the proofs of the other parts are similar).

\begin{proposition}\label{lenerated}
Assume that we are given the following data:
\begin{itemize}
\item Contraction invariant hull structures $\mathcal{Q}$ and $\mathcal{R}$ on $G$ such that $\mathcal{R}$ is basic, and $\mathcal{Q}$ is $\mathcal{R}$-enhancible and either additive or translation invariant;
\item A linear topology $\tau$ on $E$ which is a $\mathcal{Q}$-topology and a $\mathcal{R}$-topology.
\end{itemize}
Then: $\tau$ is generated by a collection of balanced $\tau$-continuous $\left[0,1\right]$-valued $\mathcal{Q}\cap\mathcal{R}$-pseudo-norms.
\end{proposition}
\begin{proof}
First, it is easy to see that $\mathcal{R}\cap\mathcal{B}$ is a basic hull structure; moreover by part (ii) of Proposition \ref{enhancible}, $\mathcal{Q}$ is $\mathcal{R}\cap\mathcal{B}$-enhancible. By Proposition \ref{balt} $\tau$ is a $\mathcal{R}\cap\mathcal{B}$ topology. By part (i) of Theorem \ref{generated}, $\tau$ is generated by a collection of $\tau$-continuous $\left[0,1\right]$-valued $\mathcal{Q}\cap\mathcal{R}\cap\mathcal{B}$-pseudo-norms.
\end{proof}

We now consider the locally convex case which is much simpler and does not require most of the theory developed here. A \emph{semi-norm} is a positively homogeneous symmetric pseudo-norm. Every semi-norm is balanced and if it is real-valued, then it is Archimedean. The pointwise supremum of any collection of semi-norms is a semi-norm. Note that being a semi-norm is a strictly stronger condition than being a a balanced Archimedean $\mathcal{Q}$-pseudo-norm, for $\mathcal{Q}$ being the hull structure of convex sets, take  $\rho$ on $\R$ defined by $\rho\left(r\right)=\left|r\right|\wedge 1$. It is not hard to show that being a semi-norm cannot be described as a lower continuity with respect to a hull structure. Semi-norms generate locally convex topologies, since the balls with respect to semi-norms are convex. We will now explore the converse.

Consider the standard construction of semi-norms. For a balanced $A\subset E$, define the \emph{gauge} of $A$ by $\rho_{A}\left(e\right):=\bigwedge\left\{r>0,~ e\in rA\right\}$. It is easy to see that $\rho_{A}$ is symmetric and positively homogeneous. If $A$ is convex, then $\rho_{A}$ is subadditive: for every $r>\rho_{A}\left(e\right)$ and $s>\rho_{A}\left(f\right)$ we have that $e+f\in rA+sA\subset \left(r+s\right)A$, where the last inclusion follows from convexity and balancedness. Also, note that $\B_{\rho_{A}}\subset A\subset \Bbo_{\rho_{A}}=\bigcap\limits_{r>1}rA$. If $\tau$ is an additive topology, then according to Proposition \ref{psco} $\rho_{A}$ is $\tau$-continuous iff $A\in\tau_{0}$. If $A\in\mathcal{Q}$, and $\mathcal{Q}$ is \emph{homothety invariant} ($Q\in\mathcal{Q}$ $\Rightarrow$ $rQ\in\mathcal{Q}$), then $\rho_{A}$ is a $\mathcal{Q}$-pseudo-norm.

\begin{proposition}\label{convex}
Assume that $\mathcal{Q}$ is a homothety invariant hull structure on $E$ which is stable with respect to taking convex hulls. Let $\tau$ be a locally convex linear $\mathcal{Q}$-topology on $E$. Then, $\tau$ is generated by a collection of $\tau$-continuous $\mathcal{Q}$-seminorms on $E$.
\end{proposition}
\begin{proof}
Let $\mathcal{U}\subset\mathcal{Q}$ be a base for $\tau_{0}$. For $U\in\mathcal{U}$ define $U^{*}:=\conv\left(U\cap-U\right)\in\mathcal{Q}$. As $\tau$ is locally convex and linear, $\left\{U^{*}\right\}_{U\in\mathcal{U}}$ is a base for $\tau_{0}$ (see Remark \ref{hn}). Then, $\left\{\rho_{U^{*}}\right\}_{U\in\mathcal{U}}$ is a collection of $\tau$-continuous $\mathcal{Q}$-seminorms which generate $\tau$.
\end{proof}

Note that other parts of Theorem \ref{generated} also hold for the locally convex case.

\section{Applications}\label{ex}

Let $L$ be a $\vee$-semi-lattice with the smallest element $0_{L}$. We call an order preserving $\mu:L\to\left[0,+\8\right]$ a \emph{submeasure} if $\mu\left(0_{L}\right)=0$ and $\mu\left(l\vee m\right)\le \mu\left(l\right)+\mu\left(m\right)$, for every $l,m\in L$. Note, that this is equivalent to the fact that $\mu$ is a $\mathcal{L}$-pseudo-norm, for $\mathcal{L}$ being the hull structure of lower sets.\medskip

Now assume that $L$ is a Boolean algebra, and in particular $L$ is distributive. It is easy to see from distributivity that $\mathcal{L}$ is additive with respect to $\vee$. However, $L$ is also a commutative group with respect to $\vartriangle$. Note that $\mu:L\to\left[0,+\8\right]$ is a submeasure iff it is a $\mathcal{L}$-pseudo-norm with respect to $\vartriangle$. Necessity is easy to see, for sufficiency observe that $\mu\left(l\vee m\right)=\mu\left(\left(l\backslash m\right)\vartriangle m\right)\le\mu\left(l\backslash m\right)+\mu\left(m\right)\le \mu\left(l\right)+\mu\left(m\right)$, for any $l,m\in L$. Also, note that for every $l\in L$ we have $l=-l$, hence every submeasure is a symmetric pseudo-norm. The following result characterizes \emph{Fr\'echet-Nikod\'ym topologies}, i.e. a group $\mathcal{L}$-topologies with respect to $\vartriangle$. For more information on the subject see \cite{weber} and the references therein.

\begin{proposition}For a topology $\tau$ on $L$ the following conditions are equivalent:
\item[(i)] $\tau$ is a Fr\'echet-Nikod\'ym topology;
\item[(ii)] $\tau$ is an additive $\mathcal{L}$-topology with respect to $\vee$, and the complementation is a $\tau$-continuous operation;
\item[(iii)] $\tau$ is generated by submeasures.
\end{proposition}
\begin{proof}
(ii)$\Rightarrow$(i) follows from the fact that all Boolean operations can be expressed through the complementation and $\vee$. (iii)$\Rightarrow$(i) follows from the fact that submeasures are precisely symmetric $\mathcal{L}$-pseudo-norms with respect to $\vartriangle$.\medskip

(i)$\Rightarrow$(ii): It is clear that continuity of $\vartriangle$ implies continuity of complementation. Assume that $l_{p}\to 0_{L}$ and $m_{p}\to 0_{L}$. Then, as $\tau$ is an $\mathcal{L}$-topology, by Proposition \ref{1ac} we have $l_{p}\backslash m_{p}\to 0_{L}$. Since $\vartriangle$ is continuous, it follows that $l_{p}\vee m_{p}=\left(l_{p}\backslash m_{p}\right)\vartriangle m_{p}\to 0_{L}$. Thus, $\vee$ is continuous at $0_{L}$.

If $l_{p}\to l$ and $m_{p}\to m$, then due to continuity of $\vartriangle$ we have $l_{p}\vartriangle l\to 0_{L}$ and $m_{p}\vartriangle m\to 0_{L}$, hence $\left(l_{p}\vartriangle l\right)\vee\left(m_{p}\vartriangle m\right)\to 0_{L}$. Using $\left(a\vartriangle b\right)\vee\left(c\vartriangle d\right)\ge \left(a\vee c\right)\vartriangle\left(b\vee d\right)$, and the fact that $\tau$ is an $\mathcal{L}$-topology again we get $\left(l_{p}\vee m_{p}\right)\vartriangle\left(l\vee m\right)\to 0_{L}$, therefore $l_{p}\vee m_{p}\to l\vee m$. Thus, $\vee$ is continuous.\medskip

Now that equivalence of (i) and (ii) is established, we can deduce (iii) from them arguing similarly to part (i) of Theorem \ref{generated}, and using the fact that submeasures are precisely $\mathcal{L}$-pseudo-norms with respect to $\vee$.
\end{proof}

Let $F$ be a vector lattice. The algebraic and order structures of $F$ give rise to a number of hull structures on it. Aside of the previously mentioned balanced, convex and full sets let us mention a few more.\smallskip

\begin{example}\label{solid}
We call $Q\subset F$ \emph{solid} if $f\in Q$, $\left|e\right|\le\left|f\right|$ $\Rightarrow$ $e\in Q$. It is easy to see that the collection $\mathcal{D}$ of solid sets is a $1$-algebraic hull structure and $\mathcal{D}\subset \mathcal{B}$. If $Q\subset F$ is solid, and $r\in\R$, then $rQ$ is solid, hence $\mathcal{D}$ is homothety invariant (and so contraction invariant). Let us show that $\mathcal{D}$ is additive. Assume that $P,Q\subset F$ are solid, and $f\in P$, $g\in Q$, and $\left|e\right|\le\left|f+g\right|$. By Riesz Decomposition Property (see \cite[Theorem 1.10]{ab0}) there are $u,v\in F$ such that $e=u+v$, $\left|u\right|\le\left|f\right|$ (and so $u\in P$) and $\left|v\right|\le\left|g\right|$ (and so $u\in Q$). We conclude that $e=u+v\in P+Q$, and so $P+Q$ is solid. It now follows immediately from Proposition \ref{lenerated} that any \emph{locally solid} topology (i.e. a linear $\mathcal{D}$-topology) on $F$ can be generated by a collection of \emph{solid} pseudo-norms (i.e. $\mathcal{D}$-pseudo-norms; $\rho$ is solid iff $\left|e\right|\le\left|f\right|$ $\Rightarrow$ $\rho\left(\left|e\right|\right)\le\rho\left(\left|f\right|\right)$). Note that $\mathcal{D}$ is not translation invariant: $\left\{0_{F}\right\}$ is solid, but for any $f\ne0_{F}$ the translation with respect to $f$ moves into a non-solid set.
\qed\end{example}\smallskip

\begin{example}
It is easy to see that full sets on $F$ form a translation invariant hull structure. Hence, by Proposition \ref{lenerated} any \emph{locally full} topology can be generated by a collection of \emph{full} balanced pseudo-norms, i.e. balanced pseudo-norms $\rho$ with the property that $e\le f\le g$ $\Rightarrow$ $\rho\left(f\right)\le\rho\left(e\right)\vee \rho\left(g\right)$. Let us show that this hull structure is not necessarily additive. Indeed, let $F:=\R^{2}$, $P:=\left\{\left(0,0\right),\left(-1,2\right)\right\}$ and $Q:=\left\{\left(0,0\right),\left(2,-1\right)\right\}$ are full, but $P+Q=\left\{\left(0,0\right),\left(2,-1\right),\left(-1,2\right),\left(1,1\right)\right\}$ is not. Symmetrization of this example also shows that the hull structure of symmetric full sets is not additive. Note that full sets are also not preserved by balanced hull, take a full set $Q:=\left\{\left(1,2\right),\left(2,1\right)\right\}$, whose balanced hull contains $\left(\frac{1}{2},1\right)$, but not $\left(1,1\right)$.
\qed\end{example}\smallskip

Recall that any topology generates a hull structure of closed sets. Convergences which are not necessarily topological (see \cite{ectv} and the references therein) also generates a hull structure of closed sets.\smallskip

\begin{example}\label{oclosed}
A net $\left(f_{p}\right)_{p\in P}\subset F$ converges \emph{in order} to $f\in F$ (denoted $f_{p}\xrightarrow[]{\mathrm{o}}f$) if there is $G\subset F$ with $\bigwedge G=0_{F}$ such that for each $g\in G$ there is $p_{0}$ such that $\left|f_{p}-f\right|\le g$, for every $p\ge p_{0}$. We call $Q\subset F$ \emph{order closed} if it contains all existing order limits of nets in $Q$. It is easy to see that translation maps and homotheties preserve order convergence, and so preserve order closed sets. Hence, the collection $\mathcal{C}$ of order closed sets is a translation invariant homothety invariant hull structure.
\qed\end{example}\smallskip

Another class of examples comes from the locally convex case. In particular the convex hull of a solid set is solid (see \cite[Theorem 1.11]{ab0}), and so according to Proposition \ref{convex}, every locally convex locally solid topology on $F$ can be generated by a collection of solid semi-norms. Consider another example.\smallskip

\begin{example}
We call $Q\subset F$ an \emph{$M$-set} if it is solid and $Q_{+}$ is an ideal in a lattice $F_{+}$, i.e. if $e,f\in Q_{+}$ $\Rightarrow$ $e\vee f\in Q$. Note that every $M$-set is convex. Indeed, if $e,f\in Q$ and $t,s\ge 0$ are such that $t+s=1$, then $\left|te+sf\right|\le t\left|e\right|+s\left|f\right|\le \left|e\right|\vee\left|f\right|\in Q$, hence $te+sf\in Q$. It is also easy to see that the class $\mathcal{M}$ of $M$-sets is homothety invariant, and so according to Proposition \ref{convex} any linear $\mathcal{M}$-topology can be generated by a collection of $\mathcal{M}$-semi-norms, i.e. solid semi-norms $\rho$ with the property that $\rho\left(e\vee f\right)=\rho\left(e\right)\vee\rho\left(f\right)$, for any $e,f\in F_{+}$.
\qed\end{example}\smallskip

We can also consider intersections of some of the hull structures to get new ones. In particular, a \emph{Fatou} set is an order closed solid set. The hull structure $\mathcal{F}=\mathcal{C}\cap \mathcal{D}$ of Fatou sets is neither translation-invariant nor additive: clearly, translation can ruin solidness; to see non-additivity observe that if $H\subset F$ is a band (hence a Fatou set; see \cite[Section 1.2]{ab0} for more details) but not a projection band, then $H^{d}$ is also a band, but $H+H^{d}$ is a non-order closed ideal. In order to deal with this hull structure we need the full power of Proposition \ref{lenerated}. Let us first show that $\mathcal{C}$ and $\mathcal{D}$ qualify.

\begin{example}\label{fatou}
We claim that the hull structure $\mathcal{C}$ of all order closed sets is $\mathcal{D}$-enhancible, where $\mathcal{D}$ is a basic hull structure of solid sets. Let $Q\subset F$ be order closed. We have $Q_{\mathcal{D}}=\left\{f\in F,~\left[-\left|f\right|,\left|f\right|\right]\subset Q\right\}$. Let us first show that $\left(Q_{\mathcal{D}}\right)_{+}\ni f_{p}\xrightarrow[]{\mathrm{o}} f$ $\Rightarrow$ $f\in Q_{\mathcal{D}}$, i.e. $\left[-f,f\right]\subset Q$. Let $g\in\left[-f,f\right]$. For every $p$ we have that $Q\supset \left[-f_{p},f_{p}\right]\ni -f_{p}\vee g\wedge f_{p}$. As $Q$ is order closed, and lattice operations are continuous with respect to order convergence, $Q\ni -f_{p}\vee g\wedge f_{p}\xrightarrow[]{\mathrm{o}} -f\vee g\wedge f= g$ implies $g\in Q$. Therefore, $\left[-f,f\right]\subset Q$, as required. Next, if $Q_{\mathcal{D}}\ni f_{p}\xrightarrow[]{\mathrm{o}} f$, then $\left(Q_{\mathcal{D}}\right)_{+}\ni \left|f_{p}\right|\xrightarrow[]{\mathrm{o}} \left|f\right|$, hence $\left|f\right|\in Q_{\mathcal{D}}$, and thus $f\in Q_{\mathcal{D}}$.

It now follows from Proposition \ref{lenerated} that any Fatou topology (i.e. linear $\mathcal{F}$-topology) is generated by a collection of Fatou pseudo-norms (i.e. $\mathcal{F}$-pseudo-norms; a pseudo-norm $\rho$ is Fatou if it is solid and $0_{F}\le f_{p}\uparrow f$ $\Rightarrow$ $\rho\left(f\right)\le\bigvee\limits_{p\in P}\rho\left(f_{p}\right)$).\medskip

Note that another way to prove this (which requires a simpler version of Proposition \ref{lenerated} with just one hull structure) is to first note that the hull structure of order closed full sets is translation invariant and homothety invariant, and moreover every locally solid topology $\tau$ is locally full (see \cite[Exercise 1, Section 2]{ab0}). Hence, if $W\in\tau_{0}\cap\mathcal{F}$, Proposition \ref{lenerated} guarantees that there is a continuous balanced $\left[0,1\right]$-valued $\tau$-continuous full pseudo-norm $\rho$ such that $\rho^{-1}\left[0,1\right)\subset W$. Define $\left|\rho\right|:F\to\left[0,1\right]$ by $\left|\rho\right|\left(f\right):=\rho\left(\left|f\right|\right)$.
Assume that $\left|e\right|\le \left|f\right|$; we also have $\left|e\right|\ge 0_{F}$, and so fullness of $\rho$ yields $\left|\rho\right|\left(e\right)=\rho\left(\left|e\right|\right)\le \rho\left(\left|f\right|\right)\vee\rho\left(0_{F}\right) =\left|\rho\right|\left(f\right)$. If particular, for $f,g\in F$ we have $$\left|\rho\right|\left(f+g\right)\le \left|\rho\right|\left(\left|f\right|+\left|g\right|\right)=\rho\left(\left|f\right|+\left|g\right|\right)\le \rho\left(\left|f\right|\right)+\rho\left(\left|g\right|\right)=\left|\rho\right|\left(f\right)+\left|\rho\right|\left(g\right).$$ Thus, $\rho$ is a solid pseudo-norm. To show that it is Fatou note that it is a composition of a $\mathcal{Q}$-pseudo-norm $\rho$ and an order continuous map $\left|\cdot\right|$. Finally, $\left|\rho\right|\left(f\right)<1$ $\Rightarrow$ $\left|f\right|\in W$ $\Rightarrow$ $f\in W$, and so $\left|\rho\right|\left[0,1\right)\subset W$. As $W$ was arbitrary, the claim follows.
\qed\end{example}\bigskip

\textbf{Acknowledgements.} The author thanks Jochen Wengenroth for bringing the author's attention to the book \cite{aek} and the service \href{mathoverflow.com/}{MathOverflow} which made it possible. The author was supported by Pacific Institute for the Mathematical Sciences.

\end{document}